\documentclass[12pt,leqno]{amsart}
\usepackage{amsmath,amsfonts,amssymb,amscd,amsthm,amsbsy,upref}
\numberwithin{equation}{section}
\newtheorem{thm}{Theorem}[section]
\newtheorem{lem}[thm]{Lemma}
\newtheorem{cor}[thm]{Corollary}

\theoremstyle{definition}
\newtheorem{defn}{Definition}[section]

\begin{document}
\title{Unconditionally $p$-converging operators and Dunford-Pettis Property of order $p$}
\author{Dongyang Chen}
\address{School of Mathematical Sciences\\ Xiamen University,
Xiamen,361005,China}
\email{cdy@xmu.edu.cn}
\author{J. Alejandro Ch\'{a}vez-Dom\'{i}nguez}
\address{Department of Mathematics, University of Oklahoma, Norman, OK 73019-3103,USA}
\email{jachavezd@math.ou.edu}
\author{Lei Li}
\address{School of Mathematical Sciences and LPMC, Nankai University, Tianjin, 300071, China}
\email{leilee@nankai.edu.cn}

\thanks{Dongyang Chen's project was supported by the Natural Science Foundation of Fujian Province of China(No. 2015J01026).
Lei Li is the corresponding author and his research is partly supported by the NSF of China (11301285)}

\begin{abstract}
In the present paper we study unconditionally $p$-converging operators and Dunford-Pettis property of order $p$. New characterizations of unconditionally $p$-converging operators and Dunford-Pettis property of order $p$ are established. Six quantities are defined to measure how far an operator is from being unconditionally $p$-converging. We prove quantitative versions of relationships of completely continuous operators,unconditionally $p$-converging operators and unconditionally converging operators. We further investigate possible quantifications of the Dunford-Pettis property of order $p$.
\end{abstract}

\date{\today}
\maketitle

\baselineskip=18pt 	
\section{Introduction and notations}
Throughout the paper, $p^{*}$ denotes the conjugate number of $p$ for $1\leq p<\infty$; if $p=1$, $l_{p^{*}}$ plays the role of $c_{0}$. $X,Y$ will denote real (or complex) Banach spaces and $\mathcal{L}(X,Y)$ the space of all the operators (=continuous linear maps) between $X$ and $Y$. $\mathcal{K}(X,Y)$ denotes the space of all the compact operators between $X$ and $Y$.
Let $X$ be a Banach space, $1\leq p<\infty$ and we denote $l_{p}(X)$ by the set of all $p$-summable sequences in $X$ with the natural norm $\|(x_{n})_{n}\|_{p}=(\sum_{n=1}^{\infty}\|x_{n}\|^{p})^{\frac{1}{p}}$.
Let $l^{w}_{p}(X)$ be the set of all weakly $p$-summable sequences in $X$. Then $l^{w}_{p}(X)$ is a Banach space with the norm $$\|(x_{n})_{n}\|_{p}^{w}=\sup\{(\sum_{n=1}^{\infty}|<x^{*},x_{n}>|^{p})^{\frac{1}{p}}:x^{*}\in B_{X^{*}}\}, \quad\forall\,(x_{n})_{n}\in l^{w}_{p}(X).$$
It is a well-known result of A. Grothendieck (\cite{G1},[12,Proposition 2.2])that the canonical correspondence $T \mapsto (Te_{n})_{n}$ provides an isometric isomorphism of $\mathcal{L}(l_{p^{*}},X)$ onto $l^{w}_{p}(X)$. A sequence $(x_{n})_{n}\in l^{w}_{p}(X)$ is \textit{unconditionally $p$-summable} if $$\sup\{(\sum_{n=m}^{\infty}|<x^{*},x_{n}>|^{p})^{\frac{1}{p}}:x^{*}\in B_{X^{*}}\}\rightarrow 0\quad\text{as}\; m\rightarrow \infty.$$
We denote the set of all unconditionally $p$-summable sequences on $X$ by $l^{u}_{p}(X)$. It is obvious that $(x_{n})_{n}$ is unconditionally $1$-summable if and only if $(x_{n})_{n}$ is unconditionally summable. J. H. Fourie and J. Swart proved that the same correspondence $T \mapsto (Te_{n})_{n}$ provides an isometric isomorphism of $\mathcal{K}(l_{p^{*}},X)$ onto $l^{u}_{p}(X)$ (see \cite{FS}). Let us recall that an operator $T:X\rightarrow Y$ is \textit{unconditionally converging} if $T$ takes weakly $1$-summable sequences to unconditionally $1$-summable sequences.
For $p=\infty$, the space $l^{u}_{\infty}(X)$ is identical to $c_{0}(X)$, the space of all norm null sequences in $X$.
Henceforth, for $p=\infty$, we refer to consider the space $c^{w}_{0}(X)$ of weakly null sequences in $X$, instead of $l^{w}_{\infty}(X)=l_{\infty}(X)$. Recall that an operator $T:X\rightarrow Y$ is \textit{completely continuous} if $T$ takes weakly null sequences to norm null sequences.
It is well-known that $p$-summing operators are precisely those operators which take weakly $p$-summable sequences(unconditionally $p$-summable sequences) to $p$-summable sequences. A natural question arises: what are operators which take weakly $p$-summable sequences to unconditionally $p$-summable sequences? This is the starting point of our investigation. The paper is organized as follows:

In Section 2, we introduce the concept of unconditionally $p$-converging operators$(1\leq p\leq \infty)$, which is the extension of unconditionally converging operators and completely continuous operators. It is proved that unconditionally $p$-converging operators coincide with the $p$-converging operators introduced by J. M. F. Castillo and F. S\'{a}nchez in \cite{CS2} although their original definitions are different. New concepts of weakly $p$-Cauchy sequences and weakly $p$-limited sets are introduced to characterize unconditionally $p$-converging operators. We establish characterizations of weakly $p$-limited sets and investigate connections between weakly $p$-limited sets and relatively norm compact sets.
A counterexample is constructed to show that an operator is unconditionally $p$-converging not precisely when its second adjoint is.

Section 3 is concerned with Dunford-Pettis property of order $p$ ($DPP_{p}$ for short) introduced in \cite{CS2}, which is a generalization of the classical Dunford-Pettis property. It turns out that many classical spaces failing Dunford-Pettis property enjoy $DPP_{p}$, such as Hardy space $H^{1}$ and Lorentz function spaces $\Lambda(W,1)$.
In this section, we use weakly $p$-Cauchy sequences and weakly $p$-limited sets to characterize $DPP_{p}$.
New characterizations of $DPP_{p}$ in dual spaces are obtained. We also introduce the notion of hereditary Dunford-Pettis property of order $p$ and establish its characterizations. In particular, we prove that a Banach space $X$ has the hereditary $DPP_{p}$ if and only if every weakly $p$-summable sequence in $X$ admits a weakly $1$-summable subsequence.
Finally, the surjective Dunford-Pettis property of order $p$, a formally weaker property than $DPP_{p}$, is introduced and its characterizations are obtained.

In the last two sections of the present paper we investigate possibilities of quantifying unconditionally $p$-converging operators and the Dunford-Pettis property of order $p$. This is inspired by a large number of recent results on quantitative versions of various theorems and properties of Banach spaces (see [1,3,13,17,18,19]).
Section 4 contains quantitative versions of the implications among three classes of operators-completely continuous,unconditionally $p$-converging and unconditionally converging ones. M. Ka\v{c}ena, O. F. K. Kalenda and J. Spurn\'{y} have already defined a quantity measuring how far an operator is from being completely continuous in \cite{KKS}. In this section, we define another equivalent quantity measuring complete continuity of an operator. We further define six  quantities measuring how far an operator is from being unconditionally $p$-converging. Moreover, we show that one of the six new quantities is equal to the quantity defined in \cite{K} to measure how far an operator is unconditionally converging in case of $p=1$.

In Section 5 we introduce a new locally convex topology and give two topological characterizations of Dunford-Pettis property of order $p$. Using the introduced quantity measuring unconditional $p$-convergence of an operator and the new locally convex topology, we show that the Dunford-Pettis property of order $p$ is automatically quantitative in a sense. We also define two quantities measuring how far a set is weakly $p$-limited. One of the two new quantities is used to quantify the Dunford-Pettis property of order $p$. The other is used to define a stronger quantitative version of Dunford-Pettis property of order $p$. Several characterizations of this quantitative version of Dunford-Pettis property of order $p$ are established.

The reader is referred to \cite{DJT} and \cite{LT} for any unexplained notation or terminology.



\section{Unconditionally $p$-converging operators}

\begin{defn}
Let $1\leq p\leq \infty$. We say that an operator $T:X\rightarrow Y$ is \textit{unconditionally $p$-converging} if $T$ takes a weakly $p$-summable sequence $(x_{n})_{n}\in l^{w}_{p}(X)((x_{n})_{n}\in c^{w}_{0}(X)$ for $p=\infty$) to an unconditionally $p$-summable sequence $(Tx_{n})_{n}\in l^{u}_{p}(Y)((x_{n})_{n}\in c_{0}(Y)$ for $p=\infty$).
\end{defn}

We begin with a simple, but extremely useful, characterization of unconditionally $p$-converging operators.

\begin{thm}\label{3.1}
Let $1\leq p<\infty$. The following are equivalent for an operator $T:X\rightarrow Y$:
\item[(1)]$T$ is unconditionally $p$-converging;
\item[(2)]$TS$ is compact for any operator $S\in \mathcal{L}(l_{p^{*}},X)$($\mathcal{L}(c_{0},X)$ for $p=1$).
\end{thm}
\begin{proof}
$(1)\Rightarrow(2)$. Let $S\in \mathcal{L}(l_{p^{*}},X)(1<p<\infty)$($\mathcal{L}(c_{0},X)$ for $p=1$). By the ideal property of unconditionally $p$-converging operators, $TS$ is unconditionally $p$-converging. Since $(e_{n})_{n}$ is weakly $p$-summable in $l_{p^{*}}(1<p<\infty)$($c_{0}$ for $p=1$), $(TSe_{n})_{n}$ is unconditionally $p$-summable.
Then there exists a compact operator $R:l_{p^{*}} \rightarrow X$ such that $Re_{n}=TSe_{n}(n=1,2,...)$. Thus $TS$ is compact.

$(2)\Rightarrow(1)$. Let $(x_{n})_{n}\in l^{w}_{p}(X)$. Then there exists an operator $S:l_{p^{*}} \rightarrow X(1<p<\infty)$($S:c_{0}\rightarrow X$ for $p=1$) such that
$Se_{n}=x_{n}(n=1,2,...)$. By (2), we get $(TSe_{n})_{n}$ is unconditionally $p$-summable. Thus $TS$ is unconditionally $p$-converging.

\end{proof}

Before another frequently useful characterization of unconditionally $p$-converging operators is given, we recall the notion of weakly $p$-convergent sequences introduced in \cite{CS}. A sequence $(x_{n})_{n}$ in a Banach space $X$ is said to be weakly $p$-convergent to $x\in X(1\leq p\leq \infty)$ if the sequence $(x_{n}-x)_{n}$ is weakly $p$-summable in $X$. Weakly $\infty$-convergent sequences are simply the weakly convergent sequences. It is natural to generalize weakly Cauchy sequences to the general case $1\leq p\leq \infty$.

\begin{defn}
Let $1\leq p\leq \infty$. We say that a sequence $(x_{n})_{n}$ in a Banach space $X$ is \textit{weakly $p$-Cauchy} if for each pair of strictly increasing sequences $(k_{n})_{n}$ and $(j_{n})_{n}$ of positive integers, the sequence $(x_{k_{n}}-x_{j_{n}})_{n}$ is weakly $p$-summable in $X$.
\end{defn}
Obviously, every weakly $p$-convergent sequence is weakly $p$-Cauchy, and the weakly $\infty$-Cauchy sequences are precisely the weakly Cauchy sequences.

\begin{thm}\label{3.2}
Let $1\leq p\leq \infty$. The following statements about an operator $T:X\rightarrow Y$ are equivalent:
\item[(1)]$T$ is unconditionally $p$-converging;
\item[(2)]$T$ sends weakly $p$-convergent sequences onto norm convergent sequences;
\item[(3)]$T$ sends weakly $p$-Cauchy sequences onto norm convergent sequences.
\end{thm}
\begin{proof}
$(1)\Rightarrow (2)$. Suppose that $(x_{n})_{n}$ is weakly $p$-convergent in $X$. We may assume that $(x_{n})_{n}$ is weakly $p$-summable.
Then there exists an operator $S:l_{p^{*}} \rightarrow X, 1<p<\infty$($S:c_{0}\rightarrow X$ for $p=1$) such that
$Se_{n}=x_{n}(n=1,2,...)$. By Theorem \ref{3.1}, $TS$ is compact and hence $(TSe_{n})_{n}$ is relatively compact. Consequently, $\lim_{n\rightarrow \infty}\|TSe_{n}\|=0$.

$(2)\Rightarrow (3)$. Let $(x_{n})_{n}$ be a weakly $p$-Cauchy sequence in $X$. By (2), for each pair of strictly increasing sequences $(k_{n})_{n}$ and $(j_{n})_{n}$ of positive integers, the sequence $(Tx_{k_{n}}-Tx_{j_{n}})_{n}$ converges to $0$ in norm and hence $(Tx_{n})_{n}$ converges in norm.

$(3)\Rightarrow (1)$. Suppose that $T$ is not unconditionally $p$-converging. By Theorem \ref{3.1}, the operator $TS$ is non-compact for some operator $S\in \mathcal{L}(l_{p^{*}},X)(1<p<\infty)$($\mathcal{L}(c_{0},X)$ for $p=1$). Then there exists a weakly null sequence $(z_{n})_{n}$ in $l_{p^{*}}(1<p<\infty)$($c_{0}$ for $p=1$) such that $\|TSz_{n}\|>\epsilon_{0}>0(n=1,2,...)$. By passing to subsequences, we may assume that the sequence $(z_{n})_{n}$ is equivalent to the unit vector basis $(e_{n})_{n}$ in $l_{p^{*}}$. Let $R: l_{p^{*}}\rightarrow l_{p^{*}}$ be an isomorphic embedding with $Re_{n}=z_{n}(n=1,2,...)$. Let $x_{n}=SRe_{n}$. Then $(x_{n})_{n}$ is weakly $p$-summable in $X$ and hence weakly $p$-Cauchy. By the assumption, $(Tx_{n})_{n}$ converges to $0$ in norm, but $\|Tx_{n}\|>\epsilon_{0}>0(n=1,2,...)$, which is a contradiction.

\end{proof}
It should be noted that Theorem \ref{3.2}(2) is the definition of the so called $p$-converging operators defined by J. M. F. Castillo and F. S\'{a}nchez in \cite{CS2}. In this note, we use the terminology unconditionally $p$-converging operators instead of $p$-converging operators.

Recall that a subset $K$ of a Banach space $X$ is \textit{relatively weakly $p$-compact} ($1\leq p<\infty$) if
$K$ is contained in $S(B_{l_{p^{*}}})$ for $1<p<\infty(S(B_{c_{0}})$ for $p=1$) for some operator $S$ from $l_{p^{*}}(c_{0}$ for $p=1$) into $X$ (see \cite{SK}).
A subset $K$ of a Banach space $X$ is said to be \textit{relatively weakly $p$-precompact} if every sequence in $K$ admits a weakly $p$-convergent subsequence (see \cite{CS1}).
Bessaga-Pe{\l}czy\'{n}ski Selection Principle yields that every relatively weakly $p$-compact set is relatively weakly $p$-precompact for any $1<p<\infty$. But the converse needs not to be true. Let $X=(\sum_{n=1}^{\infty}l^{n}_{1})_{p^{*}}(1<p<\infty)$. It follows from Bessaga-Pe{\l}czy\'{n}ski Selection Principle that $B_{X}$ is relatively weakly $p$-precompact. But $B_{X}$ is not relatively weakly $p$-compact because $X$ is not isomorphic to a quotient of $l_{p^{*}}$. Another counterexample is $L_{p}(1<p<\infty,p\neq 2)$. For each $1<p<\infty,p\neq 2$, $B_{L_{p}}$ is relatively weakly $r$-precompact, where $r=\max(p^{*},2)$, but is not relatively weakly $r$-compact because such $L_{p}$ is not isomorphic to a quotient of $l_{r^{*}}$.

By using the weakly $p$-Cauchy sequences, we can correspondingly define the conditionally weakly $p$-compact sets as follows:
\begin{defn}
Let $1\leq p\leq \infty$. We say that a subset $K$ of a Banach space $X$ is \textit{conditionally weakly $p$-compact} if every sequence in $K$ admits a weakly $p$-Cauchy subsequence.
\end{defn}

The following result,which follows from Theorem \ref{3.2}, says that unconditionally $p$-converging operators are precisely those operators that send conditionally weakly $p$-compact subsets onto relatively norm compact subsets.

\begin{thm}\label{3.7}
Let $T\in \mathcal{L}(X,Y)$ and $1\leq p<\infty$. The following statements are equivalent:
\item[(1)]$T$ is unconditionally $p$-converging;
\item[(2)]$T$ maps relatively weakly $p$-precompact subsets onto relatively norm compact subsets;
\item[(3)]$T$ maps conditionally weakly $p$-compact subsets onto relatively norm compact subsets;
\item[(4)]$T$ maps relatively weakly $p$-compact subsets onto relatively norm compact subsets.
\end{thm}

\begin{defn}
Let $X$ be a Banach space and $1\leq p<\infty$. We say that a bounded subset $K$ of $X^{*}$ is \textit{weakly $p$-limited} if
$\lim_{n\rightarrow \infty}\sup_{x^{*}\in K}|<x^{*},x_{n}>|=0$ for every $(x_{n})_{n}\in l^{w}_{p}(X)$.
\end{defn}
The following result, an immediate consequence of Theorem \ref{3.2}, is a characterization of unconditionally $p$-converging operators in terms of weakly $p$-limited subsets.

\begin{thm}\label{3.10}
Let $1\leq p<\infty$. The following are equivalent for an operator $T:X\rightarrow Y$:
\item[(1)]$T$ is unconditionally $p$-converging;
\item[(2)]$T^{*}$ maps bounded subsets of $Y^{*}$ onto weakly $p$-limited subsets of $X^{*}$.
\end{thm}

J.M.F.Castillo and F.S\'{a}nchez said that a Banach space $X\in W_{p}(1\leq p<\infty)$ if any bounded sequence in $X$ admits a weakly $p$-convergent subsequence (see \cite{CS}). We use this notion to characterize weakly $p$-limited sets.

\begin{thm}\label{3.8}
Let $1<p<\infty$ and $X$ be a Banach space. The following statements are equivalent about a bounded subset $K$ of $X^{*}$:
\item[(1)]$K$ is weakly $p$-limited;
\item[(2)]For all spaces $Y\in W_{p}$ and for every operator $T$ from $Y$ into $X$, the subset $T^{*}(K)$ is relatively norm compact;
\item[(3)]For every operator $T$ from $l_{p^{*}}$ into $X$, the subset $T^{*}(K)$ is relatively norm compact.
\end{thm}
\begin{proof}
$(1)\Rightarrow (2)$. Let $T$ be an operator from $Y\in W_{p}$ into $X$ such that $T^{*}(K)$ is not relatively norm compact.
Then there exists a sequence $(x^{*}_{n})_{n}$ in $K$ such that $(T^{*}x^{*}_{n})_{n}$ admits no norm convergent subsequences.
Since $Y^{*}$ is reflexive, by passing to a subsequence if necessary we may assume that $(T^{*}x^{*}_{n})_{n}$ converges weakly to some $y^{*}\in Y^{*}$ and $\|T^{*}x^{*}_{n}-y^{*}\|>\epsilon_{0}$ for some $\epsilon_{0}>0$ and for all $n\in \mathbb{N}$. For each $n$, choose $y_{n}$ with $\|y_{n}\|\leq 1$ such that $|<T^{*}x^{*}_{n}-y^{*},y_{n}>|>\epsilon_{0}$. Since $Y\in W_{p}$, by passing to a subsequence again if necessary one can assume that the sequence $(y_{n})_{n}$ is weakly $p$-convergent to some $y\in Y$. Thus, by hypothesis, we get $\lim_{n\rightarrow \infty}\sup_{x^{*}\in K}|<x^{*},Ty_{n}-Ty>|=0$. Note that, for each $n\in \mathbb{N}$,
$$|<T^{*}x^{*}_{n}-y^{*},y_{n}>|\leq |<x^{*}_{n},Ty_{n}-Ty>|+|<x^{*}_{n},Ty>-<y^{*},y>|+|<y^{*},y-y_{n}>|.$$
This implies that $\lim_{n\rightarrow \infty}<T^{*}x^{*}_{n}-y^{*},y_{n}>=0$, which is a contradiction.

$(2)\Rightarrow (3)$ is immediate because $l_{p^{*}}\in W_{p}$;

$(3)\Rightarrow (1)$. Let $(x_{n})_{n}\in l^{w}_{p}(X)$. Then there exists an operator $T$ from $l_{p^{*}}$ into $X$ such that $Te_{n}=x_{n}$ for all $n\in \mathbb{N}$. It follows from (3) that $T^{*}(K)$ is relatively norm compact. By the well-known characterization of relatively norm compact subsets of $l_{p}$, one can derive that $\lim_{n\rightarrow \infty}\sup_{x^{*}\in K}|<x^{*},x_{n}>|=0$.

\end{proof}

By Theorem \ref{3.8}, we see that relatively norm compact sets are weakly $p$-limited. But Theorem \ref{3.10} demonstrates that there are many weakly $p$-limited sets which are not relatively norm compact. Indeed, for each $1<p<\infty$ and for each $1<r<p^{*}$, the identity map $I_{r}$ on $l_{r}$ is unconditionally $p$-converging and hence the unit ball $B_{l_{r^{*}}}$ of $l_{r^{*}}$ is weakly $p$-limited.
In the following result, we use biorthogonal sequences to characterize weakly $p$-limited sets which are not relatively norm compact.

\begin{thm}\label{3.9}
Suppose that $X$ is reflexive and $K$ is a weakly $p$-limited subset of $X^{*}$. If $K$ is not relatively norm compact, then there exits a seminormalized biorthogonal sequence $(x_{n},x_{n}^{*})_{n}$ in $X\times (K-K)$ such that $(x_{n}^{*})_{n}$ is a basic sequence and $(x_{n})_{n}$ has no weakly $p$-Cauchy subsequence.
\end{thm}
\begin{proof}
Suppose that $K$ is not relatively norm compact, and let $(f_{n})_{n}$ be a sequence in $K$ with no norm convergent subsequence. Since $X$ is reflexive, we may assume that the sequence $(f_{n})_{n}$ converges weakly. Then there exist
two strictly increasing sequences $(k_{n})_{n}$ and $(j_{n})_{n}$ of positive integers and $\epsilon_{0}>0$ such that $\|f_{k_{n}}-f_{j_{n}}\|>\epsilon_{0}$ for all $n\in \mathbb{N}$. Let $x^{*}_{n}=f_{k_{n}}-f_{j_{n}}\in (K-K)$. Then $(x^{*}_{n})_{n}$ is weakly null. By Bessaga-Pe{\l}czy\'{n}ski Selection Principle, we can assume that $(x^{*}_{n})_{n}$ is a basic sequence. Let $(x^{**}_{n})_{n}$ be the associated sequence of coefficient functionals, and for each $n\in \mathbb{N}$, let $x_{n}\in X$ be a Hahn-Banach extension of $x^{**}_{n}$ to all of $X^{*}$. Then the sequence $(x_{n},x_{n}^{*})_{n}$ is seminormalized and biorthogonal.

It remains to show that $(x_{n})_{n}$ has no weakly $p$-Cauchy subsequence. If $(y_{n})_{n}$ is a weakly $p$-Cauchy subsequence of $(x_{n})_{n}$, then $(y_{n+1}-y_{n})_{n}$ is weakly $p$-summable. Since $K$ is weakly $p$-limited, the subset $K-K$ is also weakly $p$-limited, which implies that $\lim_{n\rightarrow \infty}\sup_{k}|<x^{*}_{k},y_{n+1}-y_{n}>|=0$. This is impossible because $(x_{n},x_{n}^{*})_{n}$ is biorthogonal.

\end{proof}

A consequence of Theorem \ref{3.9} is that for any $1<p<\infty$, there exists a relatively weakly compact sequence that admits no weakly $p$-Cauchy subsequence. Moreover, it should be noted that the converse of Theorem \ref{3.9} is true. Actually, it is easy to verify that if $K$ is a subset of $X^{*}$ and the sequence $(x_{n},x_{n}^{*})_{n}$ in $X\times (K-K)$ is biorthogonal with $\sup_{n}\|x_{n}\|<\infty$, then $K$ is not relatively norm compact.

The following result shows that an operator is unconditionally $p$-converging not precisely when its second adjoint is.

\begin{thm}\label{3.5}
\item[(1)]Let $T\in \mathcal{L}(X,Y)$ and $1\leq p\leq\infty$. If $T^{**}$ is unconditionally $p$-converging, then $T$ is unconditionally $p$-converging;
\item[(2)]For each $1\leq p\leq\infty$, there exists an unconditionally $p$-converging operator $T$, but $T^{**}$ is not unconditionally $p$-converging.
\end{thm}
\begin{proof}
(1). By the ideal property of unconditionally $p$-converging operators, $J_{Y}T$ is unconditionally $p$-converging, where $J_{Y}:Y\rightarrow Y^{**}$ is the canonical mapping. Let $S\in \mathcal{L}(l_{p^{*}},X)(1<p<\infty)$($\mathcal{L}(c_{0},X)$ for $p=1$). By Theorem \ref{3.1}, $J_{Y}TS$ is compact and hence $TS$ is compact. Again by Theorem \ref{3.1}, $T$ is unconditionally $p$-converging.

(2). J. Bourgain and F. Delbaen (see \cite{BD}) constructed a Banach space $X_{BD}$ such that $X_{BD}$ has the Schur property and $X^{**}_{BD}$ is isomorphically universal for separable Banach spaces. Since $X_{BD}$ has the Schur property, every operator from $l_{p}(1<p<\infty)$ and from $c_{0}$ into $X_{BD}$ is compact. By Theorem \ref{3.1}, every operator with domain $X_{BD}$ is unconditionally $p$-converging for each $1\leq p<\infty$. In particular, the identity map $I_{X_{BD}}$ on $X_{BD}$ is unconditionally $p$-converging. But since
$X^{**}_{BD}$ is isomorphically universal for separable Banach spaces, there exists a closed subspace $X_{p^{*}}$($X_{0}$ for $p=1$) of $X^{**}_{BD}$ such that  $X_{p^{*}}$ is isomorphic to $l_{p^{*}}$ for $1<p<\infty$ ($X_{0}$ is isomorphic to $c_{0}$ for $p=1$). This implies that $I_{X_{BD}}^{**}=I_{X_{BD}^{**}}$ is not $l_{p^{*}}$-strictly singular for $1<p<\infty$ ($c_{0}$-strictly singular for $p=1$). Thus $I_{X_{BD}}^{**}=I_{X_{BD}^{**}}$ is not unconditionally $p$-converging. For $p=\infty$, the identity map $I_{X_{BD}}$ is obviously completely continuous, but $I_{X_{BD}}^{**}=I_{X_{BD}^{**}}$ is not completely continuous because $X^{**}_{BD}$ has not the Schur property.
\end{proof}

\section{Dunford-Pettis Property of order $p$}
Let us recall that a Banach space $X$ has the \textit{Dunford-Pettis property} (in short, DPP) if for every Banach space $Y$, every weakly compact operator $T:X\rightarrow Y$ is completely continuous (see \cite{G2}). An operator $T:X\rightarrow Y$ is said to be \textit{weakly compact} if $TB_{X}$ is relatively weakly compact in $Y$.
J. M. F. Castillo and F. S\'{a}nchez extended the classical Dunford-Pettis property to the general case for $1\leq p\leq \infty$ in \cite{CS2}. Let $1\leq p\leq \infty$. A Banach space $X$ is said to have the \textit{Dunford-Pettis property of $p$} (in short, $DPP_{p}$) if for every Banach space $Y$, every weakly compact operator $T:X\rightarrow Y$ is unconditionally $p$-converging. Many classical spaces failing the DPP enjoy the $DPP_{p}$. A simple observation is that if a Banach space $X$ has cotype $q<\infty$, then $X$ has the $DPP_{p}$ for any $ 1<p<q^{*}$. Thus, the classical Hardy space $H^{1}$, which fails the DPP (see \cite{D1}), has the $DPP_{p}$ for any $1<p<2$. It is known that all the Lorentz function spaces $\Lambda(W,1)$'s fail the DPP (see \cite{D1}). But there are certain positive results for $DPP_{p}$. For example, if we take $W(t)=\frac{1}{2\sqrt{t}},t\in (0,1]$, then the space $\Lambda(W,1)$ has the $DPP_{p}$ for some $1<p\leq 2$.
Another non-reflexive space failing the DPP is the interesting space $L$ built in \cite{L}. Indeed, it was shown in \cite{BCM}that even duals of $L$ fail the DPP and odd duals of $L$ fail the surjective DPP, which is genuinely weaker than the DPP. Moreover, F. Bombal, P. Cembranos and J. Mendoza proved that for any $1\leq p<\infty$,every operator from $L$ into $l_{p}$ is compact (see \cite{BCM}). This means that $L^{*}$ has the $DPP_{p}$ for any $1<p<\infty$. More examples can be found in \cite{CS2}.

Let us start with a characterization of the $DPP_{p}$ by means of weakly $p$-limited sets.

\begin{thm}\label{4.5}
Let $1<p<\infty$. A Banach space $X$ has the $DPP_{p}$ if and only if each relatively weakly compact subset of $X^{*}$ is weakly $p$-limited.
\end{thm}
\begin{proof}
The sufficient part follows immediately from Theorem \ref{3.10}. On the other hand, let $K$ be a relatively weakly compact subset of $X^{*}$. By the Davis-Figiel-Johnson-Pe{\l}czy\'{n}ski factorization lemma (see \cite{DFJP}), there exists a reflexive space $Z$, which is a linear subspace of $X^{*}$, such that the inclusion map $J:Z\rightarrow X^{*}$ is bounded and the unit ball $B_{Z}$ of $Z$ contains $K$. Since $Z$ is reflexive, there is an operator $T:X\rightarrow Z^{*}$ such that $T^{*}=J$. By the assumption, $T$ is unconditionally $p$-converging. By Theorem \ref{3.10}, the set $T^{*}(B_{Z})=J(B_{Z})=B_{Z}$ is weakly $p$-limited in $X^{*}$. Thus $K$ is also weakly $p$-limited.
\end{proof}

Let us remark that for each $1<p<\infty$, there exists a weakly $p$-limited set which is not relatively weakly compact. Indeed, we take $X=L^{*}$, where the space $L$ is built in \cite{L}. As mentioned above, the identity $I_{X}$ on $X$ is unconditionally $p$-converging for each $1<p<\infty$. It follows from Theorem \ref{3.10} that the unit ball $B_{X^{*}}$ is
weakly $p$-limited, but it is not weakly compact because the space $L$ is non-reflexive.

The following result is an internal characterization of the $DPP_{p}$. It is a refinement of [7,Proposition 3.2].

\begin{thm}\label{4.2}
Let $1<p<\infty$ and $X$ be a Banach space. The following are equivalent:
\item[(1)]$X$ has the $DPP_{p}$;
\item[(2)]Every weakly compact operator $T$ from $X$ into $c_{0}$ is unconditionally $p$-converging;
\item[(3)]$\lim_{n\rightarrow \infty}<x^{*}_{n},x_{n}>=0$, for every weakly $p$-Cauchy sequence $(x_{n})_{n}$ in $X$ and every weakly null sequence $(x^{*}_{n})_{n}$ in $X^{*}$;
\item[(4)]$\lim_{n\rightarrow \infty}<x^{*}_{n},x_{n}>=0$, for every $(x_{n})_{n}\in l^{w}_{p}(X)$ and every weakly null sequence $(x^{*}_{n})_{n}$ in $X^{*}$;
\item[(5)]$\lim_{n\rightarrow \infty}<x^{*}_{n},x_{n}>=0$, for every $(x_{n})_{n}\in l^{w}_{p}(X)$ and every weakly Cauchy sequence $(x^{*}_{n})_{n}$ in $X^{*}$.
\end{thm}

\begin{proof}
$(1)\Rightarrow (2)$ is trivial. $(2)\Rightarrow (3)$. Given a weakly $p$-Cauchy sequence $(x_{n})_{n}$ in $X$ and a weakly null sequence $(x^{*}_{n})_{n}$ in $X^{*}$. Define an operator $T:X\rightarrow c_{0}$ by $Tx=(<x^{*}_{n},x>)_{n}$. Since $(x^{*}_{n})_{n}$ converges to $0$ weakly,  $T^{*}$ is weakly compact and so is $T$.
By (2), $T$ is unconditionally $p$-converging. By Theorem \ref{3.2}, $(Tx_{n})_{n}$ converges to some $\xi=(\xi_{k})_{k}\in c_{0}$ in norm.
Let $\epsilon>0$. There exists a positive integer $N_{1}$ such that $\|Tx_{n}-\xi\|<\frac{\epsilon}{2}$ for all $n>N_{1}$. Choose another positive integer $N_{2}$ such that $|\xi_{k}|<\frac{\epsilon}{2}$ for all $k>N_{2}$.
By the definition of $T$, we have $|<x^{*}_{n},x_{n}>|<\epsilon$ for all $n>\max(N_{1},N_{2})$. Thus
$\lim_{n\rightarrow \infty}<x^{*}_{n},x_{n}>=0$.

$(3)\Rightarrow (4)$ is trivial.

$(4)\Rightarrow (5)$. If $(x_{n})_{n}$ is weakly $p$-summable in $X$ and $(x^{*}_{n})_{n}$ is weakly Cauchy in $X^{*}$, yet $(<x^{*}_{n},x_{n}>)_{n}$ does not converge to $0$. By passing to subsequences, we may assume that $|<x^{*}_{n},x_{n}>|>\epsilon_{0}$ for some $\epsilon_{0}>0$ and all $n\in \mathbb{N}$. Since $(x_{n})_{n}$ is weakly $p$-summable and in particular weakly null, there exists a subsequence $(x_{k_{n}})_{n}$ of $(x_{n})_{n}$ such that $|<x^{*}_{n},x_{k_{n}}>|<\frac{\epsilon_{0}}{2}$ for all $n\in \mathbb{N}$. Since $(x^{*}_{n})_{n}$ is weakly Cauchy, we see that $(x^{*}_{k_{n}}-x^{*}_{n})_{n}$ is weakly null. By (3), $\lim_{n\rightarrow \infty}<x^{*}_{k_{n}}-x^{*}_{n},x_{k_{n}}>=0$. This implies that $|<x^{*}_{k_{n}}-x^{*}_{n},x_{k_{n}}>|<\frac{\epsilon_{0}}{3}$ for $n$ large enough. But for such $n$'s, we have $$\epsilon_{0}<|<x^{*}_{k_{n}},x_{k_{n}}>|\leq |<x^{*}_{k_{n}}-x^{*}_{n},x_{k_{n}}>|+|<x^{*}_{n},x_{k_{n}}>|<
\frac{5\epsilon_{0}}{6}.$$

$(5)\Rightarrow (1)$. Let $T:X\rightarrow Y$ be a weakly compact operator. Let us suppose that $T$ is not unconditionally $p$-converging. Appealing again to Theorem \ref{3.2}, we obtain a weakly $p$-summable sequence $(x_{n})_{n}$ in $X$ and $\epsilon_{0}>0$ such that $\|Tx_{n}\|>\epsilon_{0}(n=1,2,...)$. Pick $y^{*}_{n}\in Y^{*}$ such that $<y^{*}_{n},Tx_{n}>=\|Tx_{n}\|$ and $\|y^{*}_{n}\|=1$ for all $n\in \mathbb{N}$. Since $T$ is weakly compact, so is $T^{*}$. Hence there is a subsequence $(y^{*}_{k_{n}})_{n}$ of $(y^{*}_{n})_{n}$ such that the sequence $(T^{*}y^{*}_{k_{n}})_{n}$ converges weakly and hence is weakly Cauchy. The assumption ensures that the sequence $(<T^{*}y^{*}_{k_{n}},x_{k_{n}}>)_{n}=(\|Tx_{k_{n}}\|)_{n}$ converges to $0$, which is a contradiction.

\end{proof}

\begin{cor}\label{4.3}
Let $1<p<\infty$. If $X^{**}$ has the $DPP_{p}$, then so is $X$.
\end{cor}

The converse of Corollary \ref{4.3} is not true. In fact, the Banach space $X=(\sum_{n}l^{n}_{2})_{c_{0}}$ enjoys the DPP, but $X^{**}=(\sum_{n}l^{n}_{2})_{l_{\infty}}$ contains a complemented copy of $l_{2}$. Since $l_{2}$ fails the $DPP_{p}$ for any $2\leq p<\infty$, $X^{**}$ also fails the $DPP_{p}$ for any $2\leq p<\infty$. In the case of the classical DPP, there is a result better than Corollary \ref{4.3}: If $X^{*}$ has the DPP, then $X$ has the DPP too (see \cite{D1}). The analogous result is not true for the $DPP_{p}$: for each $1<p<\infty$, every operator from $l_{p}$ into Tsirelson's space $T$ is compact, hence $T$ has the $DPP_{p}$ for any $1<p<\infty$. But, for each $1<p<\infty$, there is a non-compact operator from $l_{p}$ into $T^{*}$. Thus, for each $1<p<\infty$, $T^{*}$ fails the $DPP_{p}$.

\begin{cor}\label{4.4}
Suppose that a Banach space $X$ contains no copy of $l_{1}$ and let $1<p<\infty$. The following statements are equivalent:
\item[(1)]$X^{*}$ has the $DPP_{p}$;
\item[(2)]For all Banach spaces $Y$, every weakly compact operator $T:Y\rightarrow X$ has the unconditionally $p$-converging adjoint;
\item[(3)]$\lim_{n\rightarrow \infty}<x^{*}_{n},x_{n}>=0$, for every $(x^{*}_{n})_{n}\in l^{w}_{p}(X^{*})$ and every weakly Cauchy sequence $(x_{n})_{n}$ in $X$;
\item[(4)]$\lim_{n\rightarrow \infty}<x^{*}_{n},x_{n}>=0$, for every weakly $p$-Cauchy sequence $(x^{*}_{n})_{n}$ in $X^{*}$ and every weakly null sequence $(x_{n})_{n}$ in $X$;
\item[(5)]$\lim_{n\rightarrow \infty}<x^{*}_{n},x_{n}>=0$, for every $(x^{*}_{n})_{n}\in l^{w}_{p}(X^{*})$ and every weakly null sequence $(x_{n})_{n}$ in $X$.
\end{cor}
\begin{proof}
We only prove $(2)\Rightarrow (3)$ and $(5)\Rightarrow (1)$.

$(2)\Rightarrow (3)$. Assuming the contrary, we can find $(x^{*}_{n})_{n}\in l^{w}_{p}(X^{*})$ and a weakly Cauchy sequence $(x_{n})_{n}$ in $X$ such that $|<x^{*}_{n},x_{n}>|>\epsilon_{0}$ for some $\epsilon_{0}>0$ and all $n\in \mathbb{N}$. Since $(x^{*}_{n})_{n}$ is weakly null, there exists a subsequence $(x^{*}_{k_{n}})_{n}$ of $(x^{*}_{n})_{n}$ such that $|<x^{*}_{k_{n}},x_{n}>|<\frac{\epsilon_{0}}{2}$ for all $n\in \mathbb{N}$. Thus $|<x^{*}_{k_{n}},x_{n}-x_{k_{n}}>|> \frac{\epsilon_{0}}{2}$ for all $n\in \mathbb{N}$. Define an operator $S:X^{*}\rightarrow c_{0}$ by $$Sx^{*}=(<x^{*},x_{n}-x_{k_{n}}>)_{n}, \quad x^{*}\in X^{*}.$$ It is easy to check that $S^{*}e_{n}=x_{n}-x_{k_{n}}(n=1,2,...)$, where $(e_{n})_{n}$ is the unit vector basis of $l_{1}$. Thus the operator $S^{*}$ maps $l_{1}$ into $X$ and is weakly compact. By (2), the operator $S^{**}$ is unconditionally $p$-converging. Moreover, an easy verification shows that $S^{**}=S$. By Theorem \ref{3.2}, we get $\lim_{n\rightarrow \infty}\|Sx^{*}_{k_{n}}\|=0$. It follows from the definition of the operator $S$ that $\lim_{n\rightarrow \infty}|<x^{*}_{k_{n}},x_{n}-x_{k_{n}}>|=0$, which is a contradiction.

$(5)\Rightarrow (1)$. By Theorem \ref{4.2}, it is enough to verify that
for every $(x^{*}_{n})_{n}\in l^{w}_{p}(X^{*})$ and every weakly null sequence $(x^{**}_{n})_{n}$ in $X^{**}$, the sequence $(<x^{**}_{n},x^{*}_{n}>)_{n}$ converges to $0$. Now we suppose that it is false. Then, by passing to subsequences, we may assume that $|<x^{**}_{n},x^{*}_{n}>|>\epsilon_{0}$ for some $\epsilon_{0}>0$ and all $n\in \mathbb{N}$. Of course, we may also assume that $\|x^{**}_{n}\|\leq 1$ for all $n\in \mathbb{N}$. It follows from Goldstine's Theorem that for each $n\in \mathbb{N}$, there exists an $x_{n}\in B_{X}$ such that $|<x_{n}-x^{**}_{n},x^{*}_{n}>|<\frac{\epsilon_{0}}{2}$. Then
$|<x^{*}_{n},x_{n}>|>\frac{\epsilon_{0}}{2}$ for all $n\in \mathbb{N}$. By Rosenthal's Theorem, $(x_{n})_{n}$ has a weakly Cauchy subsequence, which is still denoted by $(x_{n})_{n}$. Then there exists a subsequence $(x^{*}_{k_{n}})_{n}$ of $(x^{*}_{n})_{n}$ such that $|<x^{*}_{k_{n}},x_{n}>|<\frac{\epsilon_{0}}{3}$ for all $n\in \mathbb{N}$. By (5), we get $\lim_{n\rightarrow \infty}<x^{*}_{k_{n}},x_{k_{n}}-x_{n}>=0$, which implies that $|<x^{*}_{k_{n}},x_{k_{n}}-x_{n}>|<\frac{\epsilon_{0}}{6}$ for $n$ large enough. It is easy to verify that for such $n$'s, $|<x^{*}_{k_{n}},x_{k_{n}}>|<\frac{\epsilon_{0}}{2}$. This contradiction completes the proof.

\end{proof}

\begin{defn}
Let $1<p<\infty$. We say that a Banach space $X$ has the \textit{hereditary Dunford-Pettis property of order $p$} (in short, hereditary $DPP_{p}$)if every (closed) subspace of $X$ has the $DPP_{p}$.
\end{defn}

We present a useful characterization of hereditary $DPP_{p}$. We need a J. Elton's result that can be found in \cite{D2}.

\begin{lem}\cite{D2}\label{4.7}
If $(x_{n})_{n}$ is a normalized weakly null sequence of a space $X$ such that no subsequence of it is equivalent to the unit vector basis $(e_{n})_{n}$ of $c_{0}$, then $(x_{n})_{n}$ has a subsequence $(y_{n})_{n}$ for which given any subsequence $(z_{n})_{n}$ of $(y_{n})_{n}$ and any sequence $(\alpha_{n})_{n}\overline{\in} c_{0}$ we have $\sup_{n}\|\sum_{k=1}^{n}\alpha_{k}z_{k}\|=+\infty.$
\end{lem}

\begin{thm}\label{4.11}
Let $X$ be Banach space and $1<p<\infty$. The following are equivalent:
\item[(1)]$X$ has the hereditary $DPP_{p}$;
\item[(2)]Every normalized weakly $p$-summable sequence in $X$ admits a subsequence that is equivalent to the unit vector basis of $c_{0}$;
\item[(3)]Every weakly $p$-summable sequence in $X$ admits a weakly $1$-summable subsequence;
\item[(4)]Every weakly $p$-summable sequence in $X$ admits a subsequence $(y_{n})_{n}$ such that

$\sup_{N}\|\sum_{n=1}^{N}y_{n}\|<\infty$.
\end{thm}
\begin{proof}
$(1)\Rightarrow (2)$. Let $(x_{n})_{n}$ be a normalized weakly $p$-summable sequence in $X$ such that it admits no subsequence
that is equivalent to the unit vector basis $(e_{n})_{n}$ of $c_{0}$. It follows from Lemma \ref{4.7} that $(x_{n})_{n}$ has a subsequence $(y_{n})_{n}$ as stated in Lemma \ref{4.7}. By Bessaga-Pe{\l}czy\'{n}ski Selection Principle, we may assume that $(y_{n})_{n}$ is a basic sequence. Let $X_{0}=\overline{span}\{y_{n}:n=1,2,...\}$. Let $(y^{*}_{n})_{n}\subset X^{*}_{0}$ be the coefficient functionals of  the basic sequence $(y_{n})_{n}$. For each $N$, define a projection $P_{N}:X_{0}\rightarrow X_{0}$ by $$P_{N}(y)=\sum_{n=1}^{N}<y^{*}_{n},y>y_{n}, \quad y\in X_{0}.$$ Then the projection $P_{N}$'s are uniformly bounded in operator norm. An easy verification shows that $P^{**}_{N}y^{**}=\sum_{n=1}^{N}<y^{**},y^{*}_{n}>y_{n}$ for all $y^{**}\in X_{0}^{**}$. Lemma \ref{4.7} and the uniform boundedness of the projection $P_{N}$'s imply that $(<y^{**},y^{*}_{n}>)_{n}\in c_{0}$ for all $y^{**}\in X_{0}^{**}$, that is, $(y^{*}_{n})_{n}$ is weakly null. Since $<y^{*}_{n},y_{n}>=1$ for all $n\in\mathbb{N}$, it follows from Theorem \ref{4.2} again that $X_{0}$ fails the $DPP_{p}$.

$(2)\Rightarrow (3)$ and $(3)\Rightarrow (4)$ are obvious.

$(4)\Rightarrow (1)$. Take a subspace $X_{0}$ of $X$ that fails the $DPP_{p}$. Appealing to Theorem \ref{4.2}, we obtain a weakly compact operator $T:X_{0}\rightarrow c_{0}$ which is not unconditionally $p$-converging. Applying Theorem \ref{3.2}, we get a normalized weakly $p$-summable sequence
$(x_{n})_{n}$ in $X$ such that $\|Tx_{n}\|\geq \epsilon_{0}$ for all $n\in \mathbb{N}$. Bessaga-Pe{\l}czy\'{n}ski Selection Principle allows us to assume that the sequence
$(Tx_{n})_{n}$ is equivalent to the unit vector basis $(e_{n})_{n}$ of $c_{0}$. By the weak compactness of $T$, the sequence $(x_{n})_{n}$
admits no subsequence equivalent to the unit vector basis $(e_{n})_{n}$. By Lemma \ref{4.7}, the sequence $(x_{n})_{n}$ admits a subsequence $(y_{n})_{n}$ for which given any subsequence $(z_{n})_{n}$ of $(y_{n})_{n}$, one has $\sup_{N}\|\sum_{n=1}^{N}z_{n}\|=\infty$.

\end{proof}

A direct consequence of Theorem \ref{4.11} is the following corollary:

\begin{cor}
If a Banach space $X$ has the hereditary $DPP_{p}$, then each weakly $p$-summable sequence in $X$ admits a subsequence $(x_{n})_{n}$ such that $\lim_{n\rightarrow \infty}\|\sum_{k=1}^{n}x_{k}\|/n^{\frac{1}{p^{*}}}=0$.
\end{cor}

We close this section with the surjective $DPP_{p}$, a formally weaker property than the $DPP_{p}$. By the Davis-Figiel-Johnson-Pe{\l}czy\'{n}ski's factorization theorem (see \cite{DFJP}), a Banach space $X$ has the $DPP_{p}$ if and only if for all reflexive spaces $Y$, every operator from $X$ into $Y$ is unconditionally $p$-converging. We introduce the surjective $DPP_{p}$ by imposing that every surjective operator from $X$ onto the reflexive space $Y$ is unconditionally $p$-converging. The motivation for introducing the surjective $DPP_{p}$ was to extend the surjective DPP introduced in \cite{L}.

\begin{defn}
Let $1<p<\infty$. We say that a Banach space $X$ has the \textit{surjective $DPP_{p}$} if for all reflexive spaces $Y$, every surjective operator from $X$ onto $Y$ is unconditionally $p$-converging.
\end{defn}

The following are the internal characterizations of the surjective $DPP_{p}$.

\begin{thm}\label{4.12}
The following are equivalent for a Banach space $X$ and $1<p<\infty$:
\item[(1)]$X$ has the surjective $DPP_{p}$;
\item[(2)]$\lim_{n\rightarrow \infty}<x^{*}_{n},x_{n}>=0$, for every weakly $p$-Cauchy sequence $(x_{n})_{n}$ in $X$ and every weakly null sequence $(x^{*}_{n})_{n}$ in $X^{*}$ such that $\overline{span}\{x^{*}_{n}:n=1,2,...\}$ is reflexive;
\item[(3)]$\lim_{n\rightarrow \infty}<x^{*}_{n},x_{n}>=0$, for every $(x_{n})_{n}\in l^{w}_{p}(X)$ and every weakly null sequence $(x^{*}_{n})_{n}$ in $X^{*}$ such that $\overline{span}\{x^{*}_{n}:n=1,2,...\}$ is reflexive;
\item[(4)]$\lim_{n\rightarrow \infty}<x^{*}_{n},x_{n}>=0$, for every $(x_{n})_{n}\in l^{w}_{p}(X)$ and every weakly Cauchy sequence $(x^{*}_{n})_{n}$ in $X^{*}$ such that $\overline{span}\{x^{*}_{n}:n=1,2,...\}$ is reflexive.
\end{thm}
\begin{proof}
$(1)\Rightarrow (2)$. Let $(x_{n})_{n}\subset X$ and $(x^{*}_{n})_{n}\subset X^{*}$ be as in (2). Let $Z=\overline{span}\{x^{*}_{n}:n=1,2,...\}$. Then $(Z_{\bot})^{\bot}=Z$, where $Z_{\bot}:=\{x\in X:<x^{*},x>=0$ for all $x^{*}\in Z\}$ and $(Z_{\bot})^{\bot}:=\{x^{*}\in X^{*}:<x^{*},x>=0$ for all $x\in Z_{\bot}\}$. Let $Q:X\rightarrow X/Z_{\bot}$ be the natural quotient. Then $Q^{*}: (X/Z_{\bot})^{*}\rightarrow Z$ is a surjective isometrical isomorphism. Let $Q^{*}f_{n}=x^{*}_{n}, f_{n}\in (X/Z_{\bot})^{*}$ for all $n\in \mathbb{N}$. By (1), the quotient $Q$ is unconditionally $p$-converging. By Theorem \ref{3.2}, the sequence $(Qx_{n})_{n}$ converges in norm to $Qx$ for some $x\in X$. Thus $$|<x^{*}_{n},x_{n}-x>|=|<f_{n},Qx_{n}-Qx>|\leq (\sup_{n}\|f_{n}\|)\|Qx_{n}-Qx\|\rightarrow 0 \quad(n\rightarrow \infty).$$
Since $(x^{*}_{n})_{n}$ is weakly null, $\lim_{n\rightarrow \infty}<x^{*}_{n},x>=0$.Therefore we have $\lim_{n\rightarrow \infty}<x^{*}_{n},x_{n}>=0$.

$(2)\Rightarrow (3)$ is obvious.

$(3)\Rightarrow (4)$. Suppose that (4) is false. Then there exist a sequences $(x_{n})_{n}\in l^{w}_{p}(X)$ and a weakly Cauchy sequence $(x^{*}_{n})_{n}$ in $X^{*}$ such that $\overline{span}\{x^{*}_{n}:n=1,2,...\}$ is reflexive so that $|<x^{*}_{n},x_{n}>|>\epsilon_{0}>0$ for all $n\in \mathbb{N}$. Since the sequence $(x_{n})_{n}$ converges to $0$ weakly, there is a subsequence $(x_{k_{n}})_{n}$ of $(x_{n})_{n}$ such that $|<x^{*}_{n},x_{k_{n}}>|<\frac{\epsilon_{0}}{2}$ for all $n\in \mathbb{N}$. Since the space $\overline{span}\{x^{*}_{n}:n=1,2,...\}$ is reflexive, the space $\overline{span}\{x^{*}_{n}-x^{*}_{k_{n}}:n=1,2,...\}$ is reflexive too. By the hypothesis,
$\lim_{n\rightarrow \infty}<x^{*}_{n}-x^{*}_{k_{n}},x_{k_{n}}>=0$. Thus, $|<x^{*}_{n}-x^{*}_{k_{n}},x_{k_{n}}>|<\frac{\epsilon_{0}}{2}$ for $n$ large enough, which implies that for such $n$'s,
$|<x^{*}_{k_{n}},x_{k_{n}}>|<\epsilon_{0}$, a contradiction.

$(4)\Rightarrow (1)$. Suppose that $X$ fails the surjective $DPP_{p}$. Then there exists a surjective operator $T$ from $X$ onto a reflexive space $Y$ such that $T$ is not unconditionally $p$-converging. By Theorem \ref{3.2}, there exists a normalized weakly $p$-summable sequence $(x_{n})_{n}$ in $X$ such that $\|Tx_{n}\|>\epsilon_{0}$ for all $n\in \mathbb{N}$. For each $n$, choose $y^{*}_{n}\in Y^{*}$ with $\|y^{*}_{n}\|=1$ such that $<y^{*}_{n},Tx_{n}>=\|Tx_{n}\|$. By the reflexivity of $Y$, we may assume that the sequence $(y^{*}_{n})_{n}$ converges to $0$ weakly by passing to subsequences if necessary.
Let $x^{*}_{n}=T^{*}y^{*}_{n}$.
Then the sequence $(x^{*}_{n})_{n}$ converges to $0$ weakly too. Since $T$ is surjective, the operator $T^{*}:Y^{*}\rightarrow X^{*}$ is an isomorphic embedding. This implies that the space $\overline{span}\{x^{*}_{n}:n=1,2,...\}$ is contained in $T^{*}(\overline{span}\{y^{*}_{n}:n=1,2,...\})$ and hence is reflexive. By (4), $\lim_{n\rightarrow \infty}<x^{*}_{n},x_{n}>=0$, a contradiction because $<x^{*}_{n},x_{n}>>\epsilon_{0}$ for all $n\in\mathbb{N}$. This concludes the proof.

\end{proof}

An immediate consequence of Theorem \ref{4.12} is the following:

\begin{cor}\label{4.13}
Let $1<p<\infty$. If $X^{**}$ has the surjective $DPP_{p}$, then so is $X$.
\end{cor}
We also use the space $X=(\sum_{n}l^{n}_{2})_{c_{0}}$ to show that the converse of Corollary \ref{4.13} is not true.
The same argument shows that the space $X=(\sum_{n}l^{n}_{2})_{c_{0}}$ enjoys the surjective $DPP_{p}$ for any $1<p<\infty$, but $X^{**}$ also fails the surjective $DPP_{p}$ for any $2\leq p<\infty$.

The following result analogous to Theorem 3 in \cite{BCM} shows that the surjective $DPP_{p}$ and the $DPP_{p}$ coincide for certain classes of Banach spaces.

\begin{thm}
If a Banach space $X$ contains a complemented copy of $l_{1}$, then $X$ has the $DPP_{p}$ if and only if $X$ has the surjective $DPP_{p}$.
\end{thm}

\section{Quantifying unconditionally $p$-converging operators}
As discussed above, we see that unconditionally $p$-converging operators are intermediate between completely continuous operators and unconditionally converging operators. Precisely, we have the following implications:
\begin{center}
$T$ completely  continuous $\Rightarrow T$ unconditionally $p$-converging $\Rightarrow T$ unconditionally converging.
\end{center}
In this section, we quantify these implications. We need some necessary quantities.

Let $(x_{n})_{n}$ be a bounded sequence in a Banach space $X$. Set $$ca((x_{n})_{n})=\inf_{n}\sup\{\|x_{k}-x_{l}\|:k,l\geq n\}.$$ This quantity is a measure of non-Cauchyness of the sequence $(x_{n})_{n}$. More precisely, $ca((x_{n})_{n})=0$ if and only if $(x_{n})_{n}$ is norm Cauchy. In \cite{KKS}, an important quantity measuring how far an operator $T:X\rightarrow Y$ is from being completely continuous, denoted as $cc(T)$, is defined by
\begin{center}
$cc(T)=\sup\{ca((Tx_{n})_{n}): (x_{n})_{n}\subset B_{X}$ weakly Cauchy $\}$.
\end{center}

Obviously, $T$ is completely continuous if and only if $cc(T)=0$. In this note, we define another equivalent quantity measuring the complete continuity of an operator $T:X\rightarrow Y$ as follows:
\begin{center}
$cc_{n}(T)=\sup\{\limsup_{n}\|Tx_{n}\|: (x_{n})_{n}\subset B_{X}$ weakly null $\}$.
\end{center}
Obviously, $T$ is completely continuous if and only if $cc_{n}(T)=0$. The following theorem demonstrates these two quantities are equivalent.

\begin{thm}\label{2.1}
Let $T\in \mathcal{L}(X,Y)$. Then $cc_{n}(T)\leq cc(T)\leq 2cc_{n}(T)$.
\end{thm}
To prove Theorem \ref{2.1}, we need the following lemma.
\begin{lem}\label{2.2}
Let $X$ be a Banach space and $(x_{n})_{n}$ be a weakly null sequence in $B_{X}$. Let $\epsilon>0$ be such that $\|x_{n}\|>\epsilon$ for all $n\in \mathbb{N}$. Then, for every $\delta>0$, there is a subsequence $(x_{k_{n}})_{n}$ of $(x_{n})_{n}$ such that $ca((x_{k_{n}})_{n})\geq \epsilon-\delta$.
\end{lem}
\begin{proof}
We set $x_{k_{1}}=x_{1}$. Choose $x^{*}_{1}\in S_{X^{*}}$ such that $<x^{*}_{1},x_{k_{1}}>=\|x_{k_{1}}\|$. Since $(x_{n})_{n}$ is weakly null, there exists $k_{2}>k_{1}$ such that $|<x^{*}_{1},x_{k_{2}}>|<\delta$. Then
$$\|x_{k_{1}}-x_{k_{2}}\|\geq |<x^{*}_{1},x_{k_{1}}-x_{k_{2}}>|\geq |<x^{*}_{1},x_{k_{1}}>|-|<x^{*}_{1},x_{k_{2}}>|\geq\epsilon-\delta.$$
Suppose that we have obtained $\{x_{k_{1}},x_{k_{2}},...,x_{k_{n}}\}$ such that $\|x_{k_{i}}-x_{k_{n}}\|\geq \epsilon-\delta$ for $i=1,2,...,n-1$. Let $Y_{n}=span\{x_{k_{1}},x_{k_{2}},...,x_{k_{n}}\}$. Pick a $c$-net $\{z_{1},z_{2},...,z_{m}\}\subset S_{Y_{n}}$ for $S_{Y_{n}}$, where $0<c<\frac{\delta}{2}$. Choose $z^{*}_{1},z^{*}_{2},...,z^{*}_{m}$ in $S_{X^{*}}$ such that $<z^{*}_{i},z_{i}>=1$ for $i=1,2,...,m$. Since $(x_{n})_{n}$ is weakly null, there exists $k_{n+1}>k_{n}$ such that $|<z^{*}_{i},x_{k_{n+1}}>|<c$ for all $i=1,2,...,m$. Then, for each $1\leq j\leq n$, there exists $1\leq i\leq m$ such that $\|\frac{x_{k_{j}}}{\|x_{k_{j}}\|}-z_{i}\|<c$. Thus
\begin{align*}
\|x_{k_{j}}-x_{k_{n+1}}\|&\geq |<z^{*}_{i},x_{k_{j}}-x_{k_{n+1}}>|\\
&\geq 1-|<z^{*}_{i},x_{k_{n+1}}>|-|<z^{*}_{i},x_{k_{j}}-z_{i}>|\\
&\geq 1-c-\|x_{k_{j}}-z_{i}\|\\
&\geq 1-c-(1+c-\epsilon)=\epsilon-2c\\
&\geq \epsilon-\delta\\
\end{align*}
By induction, we get a subsequence $(x_{k_{n}})_{n}$ such that $\|x_{k_{n}}-x_{k_{m}}\|\geq \epsilon-\delta(n\neq m,n,m=1,2,...)$. This yields that $ca((x_{k_{n}})_{n})\geq \epsilon-\delta$.

\end{proof}

\begin{proof}[Proof of Theorem \ref{2.1}]

Step 1. $cc(T)\leq 2cc_{n}(T)$.

We may suppose that $cc(T)>0$ and fix any $c>0$ satisfying $cc(T)>c$. Then there is a weakly Cauchy sequence $(x_{n})_{n}$ in $B_{X}$ such that $ca((Tx_{n})_{n})>c$. It follows that there exist two strictly increasing sequences $(k_{n})_{n},(l_{n})_{n}$ of positive integers such that $\|Tx_{k_{n}}-Tx_{l_{n}}\|>c$ for all $n\in \mathbb{N}$. Set $z_{n}=(x_{k_{n}}-x_{l_{n}})/2$. Then $(z_{n})_{n}$ is a weakly null sequence in $B_{X}$ and $\|Tz_{n}\|>c/2$ for each $n\in \mathbb{N}$. Hence $\limsup_{n}\|Tz_{n}\|\geq c/2$ and then $cc_{n}(T)\geq c/2$. Since $c<cc(T)$ is arbitrary, we get $cc(T)\leq 2cc_{n}(T)$.

Step 2. $cc_{n}(T)\leq cc(T)$.

We may suppose that $\|T\|=1$ and $cc_{n}(T)>0$. Suppose that $cc_{n}(T)>\epsilon>0$. Then there is a weakly null sequence $(x_{n})_{n}$ in $B_{X}$ such that $\limsup_{n}\|Tx_{n}\|>\epsilon$. This yields a subsequence of $(x_{n})_{n}$, still denoted by $(x_{n})_{n}$, so that $\|Tx_{n}\|>\epsilon$ for each $n\in \mathbb{N}$. By Lemma \ref{2.2}, for every $\delta>0$, there is a subsequence $(x_{k_{n}})_{n}$ of $(x_{n})_{n}$ such that $ca((Tx_{k_{n}})_{n})\geq \epsilon-\delta$. This means that $cc(T)\geq \epsilon-\delta$. Since $\delta>0$ is arbitrary, we get $cc(T)\geq \epsilon$. By the arbitrariness of $\epsilon<cc_{n}(T)$, we obtain $cc_{n}(T)\leq cc(T)$. This completes the proof of Theorem \ref{2.1}.
\end{proof}
To quantify unconditionally $p$-converging operators, we will need two measures of non-compactness. Let us fix some notations. If $A$ and $B$ are nonempty subsets of a Banach space $X$, we set $$d(A,B)=\inf\{\|a-b\|:a\in A,b\in B\},$$$$\widehat{d}(A,B)=\sup\{d(a,B):a\in A\}.$$ Thus, $d(A,B)$ is the ordinary distance between $A$ and $B$, and $\widehat{d}(A,B)$ is the non-symmetrized Hausdorff distance from $A$ to $B$.

Let $A$ be a bounded subset of a Banach space $X$. The Hausdorff measure of non-compactness of $A$ is defined by
\begin{center}
$\chi(A)=\inf\{\widehat{d}(A,F):F\subset X$ finite$\}$,
\end{center}

\begin{center}
$\chi_{0}(A)=\inf\{\widehat{d}(A,F):F\subset A$ finite$\}$.
\end{center}
Then $\chi(A)=\chi_{0}(A)=0$ if and only if $A$ is relatively norm compact. It is easy to verify that
\begin{equation}\label{3}
\chi(A)\leq \chi_{0}(A)\leq 2\chi(A).
\end{equation}

Now we define five quantities which measure how far an operator is from being unconditionally $p$-converging. Let $T\in \mathcal{L}(X,Y)$ and $1\leq p<\infty$. We set
\begin{center}
$uc_{p}^{1}(T)=\sup\{\limsup_{n}\|Tx_{n}\|: (x_{n})_{n}\in l^{w}_{p}(X), (x_{n})_{n}\subset B_{X}\},$
\end{center}

\begin{center}
$uc_{p}^{2}(T)=\sup\{ca((Tx_{n})_{n}): (x_{n})_{n}\subset B_{X}$ weakly $p$-Cauchy $\},$
\end{center}

\begin{center}
$uc_{p}^{3}(T)=\sup\{ca((Tx_{n})_{n}): (x_{n})_{n}\subset B_{X}$ weakly $p$-convergent $\},$
\end{center}

\begin{center}
$uc_{p}^{4}(T)=\sup\{\chi_{0}(TL): L\subset B_{X}$ relatively weakly $p$-compact $\}$,
\end{center}

\begin{center}
$uc_{p}^{5}(T)=\sup\{\chi_{0}(TL): L\subset B_{X}$ relatively weakly $p$-precompact $\}$.
\end{center}

Clearly, $uc_{p}^{1}(T)=uc_{p}^{2}(T)=uc_{p}^{3}(T)=uc_{p}^{4}(T)=uc_{p}^{5}(T)=0$ if and only if $T$ is unconditionally $p$-converging. It turns out that the above five quantities are equivalent.

\begin{thm}
Let $T\in \mathcal{L}(X,Y)$ and $1<p<\infty$. Then
$$uc_{p}^{5}(T)\leq uc_{p}^{3}(T)\leq uc_{p}^{2}(T)\leq 2uc_{p}^{1}(T)\leq 2uc_{p}^{4}(T)\leq 2uc_{p}^{5}(T).$$
\end{thm}
\begin{proof}
Step 1. $uc_{p}^{5}(T)\leq uc_{p}^{3}(T)$.

We may assume that $uc_{p}^{5}(T)>0$. Let us fix any $0<c<uc_{p}^{5}(T).$ Then there exists a relatively weakly $p$-precompact subset $L\subset B_{X}$ such that $\chi_{0}(TL)>c$. By induction, we can construct a sequence $(x_{n})_{n}$ in $L$ such that $\|Tx_{n}-Tx_{m}\|>c, n\neq m, n,m=1,2,...$ Since $L$ is relatively weakly $p$-precompact, the sequence $(x_{n})_{n}$ admits a weakly $p$-convergent subsequence that is still denoted by $(x_{n})_{n}$. Thus we get $ca((Tx_{n})_{n})\geq c$, which yields $uc_{p}^{3}(T)\geq c$. By the arbitrariness of $c$, we get $uc_{p}^{5}(T)\leq uc_{p}^{3}(T)$.

Step 2. $uc_{p}^{2}(T)\leq 2uc_{p}^{1}(T)$.

We assume that $uc_{p}^{2}(T)>0$ and fix any $0<c<uc_{p}^{2}(T)$. Then there is a weakly $p$-Cauchy sequence $(x_{n})_{n}$ in $B_{X}$ such that $ca((Tx_{n})_{n})>c$. By induction, there exist two strictly increasing sequences $(k_{n})_{n},(l_{n})_{n}$ of positive integers such that $\|Tx_{k_{n}}-Tx_{l_{n}}\|>c$ for all $n\in \mathbb{N}$. Set $z_{n}=(x_{k_{n}}-x_{l_{n}})/2$. Then $(z_{n})_{n}$ is a weakly $p$-summable sequence in $B_{X}$ and $\|Tz_{n}\|>c/2$ for each $n\in \mathbb{N}$. Hence $uc_{p}^{1}(T)\geq c/2$. Since $c$ is arbitrary, we get Step 2.

Step 3. $uc_{p}^{1}(T)\leq uc_{p}^{4}(T)$.

Suppose $uc_{p}^{1}(T)>c>0$. Then there exists a weakly $p$-summable sequence $(x_{n})_{n}$ in $B_{X}$ such that $\|Tx_{n}\|>c$ for all $n\in \mathbb{N}$. We claim that $\chi_{0}((Tx_{n})_{n})\geq c$. If this is false, we can find a finite subset $F$ of $(Tx_{n})_{n}$ such that $\widehat{d}((Tx_{n})_{n},F)<c$. Since $F$ is finite, there exist $y\in F$ and a subsequence $(Tx_{k_{n}})_{n}$ of $(Tx_{n})_{n}$ such that $\|Tx_{k_{n}}-y\|\leq c$ for each $n\in \mathbb{N}$. Since the sequence $(Tx_{k_{n}})_{n}$ is weakly null, we get $\|y\|\leq c$. This contradiction completes the proof Step 3.

The remaining inequalities $uc_{p}^{3}(T)\leq uc_{p}^{2}(T), uc_{p}^{4}(T)\leq uc_{p}^{5}(T)$ are immediate.

\end{proof}

It should be mentioned that a quantity is defined in \cite{K} to measure how far an operator is unconditionally converging as follows:
$$uc(T)=\sup\{ca((\sum_{i=1}^{n}Tx_{i})_{n}): (x_{n})_{n}\in l^{w}_{1}(X),\|(x_{n})_{n}\|^{w}_{1}\leq 1\}.$$
Obviously, $uc(T)=0$ if and only if $T$ is unconditionally converging. Inspired by this quantity, we define the sixth quantity measuring how far an operator is unconditionally $p$-converging as follows:
$$uc_{p}^{6}(T)=\sup\{\limsup_{n}\|Tx_{n}\|: (x_{n})_{n}\in l^{w}_{p}(X), \|(x_{n})_{n}\|^{w}_{p}\leq 1\}.$$
It is obvious that $uc_{p}^{6}(T)=0$ if and only if $T$ is unconditionally $p$-converging.
This new quantity will be used in next section to prove a quantitative version of the Dunford-Pettis property of order $p$.
\begin{thm}\label{2.3}
Let $T\in \mathcal{L}(X,Y)$. Then $uc_{1}^{6}(T)=uc(T)$.
\end{thm}
\begin{proof}
Step 1. $uc_{1}^{6}(T)\leq uc(T)$.

Let $(x_{n})_{n}\in l^{w}_{1}(X)$ with $\|(x_{n})_{n}\|^{w}_{1}\leq 1$. It aims to show $\limsup_{n}\|Tx_{n}\|\leq ca((\sum_{i=1}^{n}Tx_{i})_{n})$. Let $c>ca((\sum_{i=1}^{n}Tx_{i})_{n})$. Then there exists $n\in \mathbb{N}$ such that $\|\sum_{i=1}^{k}Tx_{i}-\sum_{i=1}^{l}Tx_{i}\|<c$ for all $k,l\geq n$. In particular, we have  $\|Tx_{k}\|=\|\sum_{i=1}^{k}Tx_{i}-\sum_{i=1}^{k-1}Tx_{i}\|<c$ for all $k\geq n+1$. Thus one can derive that $\limsup_{n}\|Tx_{n}\|\leq c$. Since $c>ca((\sum_{i=1}^{n}Tx_{i})_{n})$ is arbitrary, we get $\limsup_{n}\|Tx_{n}\|\leq ca((\sum_{i=1}^{n}Tx_{i})_{n})$.

Step 2. $uc(T)\leq uc_{1}^{6}(T)$.

We can suppose that $uc(T)>0$ and fix an arbitrary $0<c<uc(T)$. Then there exists $(x_{n})_{n}\in l^{w}_{1}(X)$ with $\|(x_{n})_{n}\|^{w}_{1}\leq 1$ such that $ca((\sum_{i=1}^{n}Tx_{i})_{n})>c$. By induction, we can find two strictly increasing sequences $(k_{n})_{n},(l_{n})_{n},l_{n}<k_{n}$ of positive integers such that $\|\sum_{i=l_{n}+1}^{k_{n}}Tx_{i}\|>c$ for all $n\in \mathbb{N}$. Let $z_{n}=\sum_{i=l_{n}+1}^{k_{n}}x_{i}(n=1,2,...)$. It is easy to see that $(z_{n})_{n}$ belongs to $l^{w}_{1}(X)$ with $\|(z_{n})_{n}\|^{w}_{1}\leq 1$ such that $\|Tz_{n}\|>c$ for all $n\in \mathbb{N}$, which yields $\limsup_{n}\|Tz_{n}\|\geq c$. Hence $uc_{1}^{6}(T)\geq c$ and the proof of Step 2 is completed.

\end{proof}

Combining Theorem \ref{2.1} with Theorem \ref{2.3}, we get the promised quantitative versions of the above implications.

\begin{thm}
Let $T\in \mathcal{L}(X,Y)$ and $1\leq p<\infty$. Then $uc(T)\leq uc_{p}^{6}(T)\leq cc(T)$.
\end{thm}

\section{Quantifying Dunford-Pettis property of order $p$}
Let $X$ be a Banach space and let $\mathcal{F}$ be the family of all weakly compact subsets of $B_{X^{*}}$. For $F\in \mathcal{F}$, define a semi-norm $q_{F}$ on $X^{**}$ by $$q_{F}(x^{**})=\sup_{x^{*}\in F}|<x^{**},x^{*}>|, \quad x^{**}\in X^{**}.$$   The locally convex topology generated by the family of semi-norms $\{q_{F}:F\in \mathcal{F}\}$ is called the Mackey topology, denoted by $\tau(X^{**},X^{*})$. The restriction to $X$ of the Mackey topology $\tau(X^{**},X^{*})$ is called the Right topology in \cite{PVWY}. This topology is denoted by $\rho_{X}$ or simply $\rho$ when $X$ is obvious.

In this section, we introduce a new locally convex topology. Let $X$ be a Banach space and let $1\leq p<\infty$. Let $\mathcal{F}_{p}$ be the family of all relatively weakly $p$-compact subsets of $X$. For $F\in \mathcal{F}_{p}$, we define a semi-norm $q_{F}$ on $X^{*}$ by $$q_{F}(x^{*})=\sup_{x\in F}|<x^{*},x>|, \quad x^{*}\in X^{*}.$$
The locally convex topology generated by the family of semi-norms $\{q_{F}:F\in \mathcal{F}_{p}\}$ is denoted by $\rho^{*}_{p}$ when $X$ is obvious. Applying Grothendieck's Completeness Theorem([24, p.148]), we obtain that the space $(X^{*},\rho^{*}_{p})$ is complete. Hence, a bounded subset $A$ of $X^{*}$ is relatively $\rho^{*}_{p}$-compact if and only if $A$ is totally bounded, equivalently, the set $A|_{F}=\{x^{*}|_{F}:x^{*}\in A\}$ is totally bounded in $l_{\infty}(F)$ for each relatively weakly $p$-compact subset $F\subset B_{X}$. So, if we set
$$\chi_{m}^{p}(A)=\sup\{\chi_{0}(A|_{F}):F\in \mathcal{F}_{p}, F\subset B_{X}\},$$
then $A$ is relatively $\rho^{*}_{p}$-compact if and only if $\chi_{m}^{p}(A)=0$. The following result, which is immediate from [19, Lemma 4.4], implies that an operator $T:X\rightarrow Y$ is unconditionally $p$-converging if and only if $T^{*}B_{Y^{*}}$ is relatively $\rho^{*}_{p}$-compact.

\begin{thm}
Let $T\in \mathcal{L}(X,Y)$ and $1\leq p<\infty$. Then $\frac{1}{2}uc_{p}^{4}(T)\leq \chi_{m}^{p}(T^{*}B_{Y^{*}})\leq 2uc_{p}^{4}(T)$.
\end{thm}

Let $(x^{*}_{n})_{n}$ be a bounded sequence in $X^{*}$. We set
$$ca_{\mathcal{F}_{p}}((x^{*}_{n})_{n})=\sup_{F\in \mathcal{F}_{p},F\subset B_{X}}\inf_{n}\sup\{q_{F}(x^{*}_{k}-x^{*}_{l}):k,l\geq n\},$$ and
\begin{center}
$\widetilde{ca}_{\mathcal{F}_{p}}((x^{*}_{n})_{n})=\inf\{ca_{\mathcal{F}_{p}}((x^{*}_{k_{n}})_{n}):(x^{*}_{k_{n}})_{n}$ is a subsequence of $(x^{*}_{n})_{n}\}$.
\end{center}
The quantity $ca_{\mathcal{F}_{p}}$ measures how far the sequence $(x^{*}_{n})_{n}$ is from being $\rho^{*}_{p}$-Cauchy. In particular, $ca_{\mathcal{F}_{p}}((x^{*}_{n})_{n})=0$ if and only if the sequence $(x^{*}_{n})_{n}$ is $\rho^{*}_{p}$-Cauchy.

The following result contains two topological characterizations of $DPP_{p}$.

\begin{thm}\label{5.1}
The following are equivalent about a Banach space $X$ and $1<p<\infty$:
\item[(1)]$X$ has the $DPP_{p}$;
\item[(2)]Every weakly $p$-summable sequence in $X$ is $\rho$-null;
\item[(3)]Every weakly convergent sequence in $X^{*}$ is $\rho^{*}_{p}$-convergent.
\end{thm}
\begin{proof}
The equivalence of (1) and (2) is essentially Theorem \ref{4.5}.
The implication $(3)\Rightarrow (1)$ follows from Theorem \ref{4.2}. It remains to prove $(1)\Rightarrow (3)$.

Let $(x^{*}_{n})_{n}$ be weakly null in $X^{*}$. Define an operator $T:X\rightarrow c_{0}$ by $$Tx=(<x^{*}_{n},x>)_{n}, \quad x\in X.$$
Since $(x^{*}_{n})_{n}$ is weakly null, $T$ is weakly compact. By (1), we get $T$ is unconditionally $p$-converging. Let $F\in \mathcal{F}_{p}$. It follows from Theorem \ref{3.7} that $TF$ is relatively norm compact in $c_{0}$. By the well-known characterization of relatively norm compact subsets of $c_{0}$, we get $$\lim_{n\rightarrow \infty}q_{F}(x^{*}_{n})=\lim_{n\rightarrow \infty}\sup_{x\in F}|<x^{*}_{n},x>|=0,$$
which implies that $(x^{*}_{n})_{n}$ is $\rho^{*}_{p}$-null.

\end{proof}

To quantify the $DPP_{p}$, we will need several measures of weak non-compactness.
Let $A$ be a bounded subset of a Banach space $X$.
The de Blasi measure of weak non-compactness of $A$ is defined by

\begin{center}
$\omega(A)=\inf\{\widehat{d}(A,K):\emptyset\neq K\subset X$ is weakly compact $\}.$
\end{center}

Then $\omega(A)=0$ if and only if $A$ is relatively weakly compact. It is easy to verify that
\begin{center}
$\omega(A)=\inf\{\epsilon>0:$ there exists a weakly compact subset $K$ of $X$ such that $A\subset K+\epsilon B_{X}\}$.
\end{center}
Other commonly used quantities measuring weak non-compactness are:
\begin{center}
$wk_{X}(A)=\widehat{d}(\overline{A}^{w^{*}},X),$ where $\overline{A}^{w^{*}}$ denotes the $weak^{*}$ closure of $A$ in $X^{**}$.
\end{center}

\begin{center}
$wck_{X}(A)=\sup\{d(clust_{X^{**}}((x_{n})_{n}),X):(x_{n})_{n}$ is a sequence in $A\}$, where $clust_{X^{**}}((x_{n})_{n})$ is the set of all $weak^{*}$ cluster points in $X^{**}$ of $(x_{n})_{n}$.
\end{center}

\begin{center}
$\gamma(A)=\sup\{|\lim_{n}\lim_{m}<x^{*}_{m},x_{n}>-\lim_{m}\lim_{n}<x^{*}_{m},x_{n}>|:(x_{n})_{n}$ is a sequence in $A$, $(x^{*}_{m})_{m}$ is a sequence in $B_{X^{*}}$ and all the involved limits exist$\}$.
\end{center}
It follows from [1,Theorem 2.3] that for any bounded subset $A$ of a Banach space $X$ we have
$$wck_{X}(A)\leq wk_{X}(A)\leq \gamma(A)\leq 2wck_{X}(A),$$
$$wk_{X}(A)\leq \omega(A).$$
For an operator $T$, $\omega(T),wk_{Y}(T),wck_{Y}(T),\gamma(T)$ denote $\omega(TB_{X}),wk_{Y}(TB_{X}),wck_{Y}(TB_{X})$ and $\gamma(TB_{X})$,respectively.
C. Angosto and B. Cascales(\cite{AC})proved the following inequality:
\begin{center}
$\gamma(T)\leq \gamma(T^{*})\leq 2\gamma(T)$, for any operator $T$.
\end{center}
So,putting these inequalities together, we get,for any operator $T$,
\begin{equation}\label{1}
{\frac{1}{2}}wk_{Y}(T)\leq wk_{X^{*}}(T^{*})\leq 4wk_{Y}(T).
\end{equation}

Let $X$ be a Banach space and $A$ be a bounded subset of $X^{*}$. For $1\leq p<\infty$, we set
$$\iota_{p}(A)=\sup\{\limsup_{n}\sup_{x^{*}\in A}|<x^{*},x_{n}>|:(x_{n})_{n}\in l^{w}_{p}(X),(x_{n})_{n}\subset B_{X}\},$$
$$\eta_{p}(A)=\sup\{\limsup_{n}\sup_{x^{*}\in A}|<x^{*},x_{n}>|:(x_{n})_{n}\in l^{w}_{p}(X),\|(x_{n})_{n}\|^{w}_{p}\leq 1\}.$$
These two quantities measure how far $A$ is weakly $p$-limited. Obviously, $\eta_{p}(A)=\iota_{p}(A)=0$ if and only if $A$ is weakly $p$-limited. The following theorem says, in particular, that weakly $p$-limited sets coincide with relatively $\rho^{*}_{p}$-compact sets. Its proof is similar to [19, Lemma 5.6].
\begin{thm}
Let $X$ be a Banach space, $1\leq p<\infty$ and $A$ be a bounded subset of $X^{*}$. Then
$$\frac{1}{8}\chi_{m}^{p}(A)\leq \iota_{p}(A)\leq \chi_{m}^{p}(A).$$
\end{thm}

In the following theorem, we quantify the $DPP_{p}$ by using the quantities $\omega(\cdot)$, $\iota_{p}(\cdot)$, $\widetilde{ca}_{\mathcal{F}_{p}}(\cdot)$ and $\chi_{m}^{p}(\cdot)$.

\begin{thm}
Let $X$ be a Banach space and $1<p<\infty$. The following are equivalent:
\item[(1)]$X$ has the $DPP_{p}$;
\item[(2)]$uc_{p}^{1}(T)\leq \omega(T^{*})$ for every operator $T$ from $X$ into any Banach space $Y$;
\item[(3)]$\iota_{p}(A)\leq \omega(A)$ for every bounded subset $A$ of $X^{*}$;
\item[(4)]$\widetilde{ca}_{\mathcal{F}_{p}}((x^{*}_{n})_{n})\leq 2\omega((x^{*}_{n})_{n})$ whenever $(x^{*}_{n})_{n}$ is a bounded sequence in $X^{*}$;
\item[(5)]$\chi_{m}^{p}(A)\leq 2\omega(A)$ for every bounded subset $A$ of $X^{*}$.
\end{thm}
\begin{proof}
$(2)\Rightarrow (1)$ is obvious. $(3)\Rightarrow (1)$ and $(5)\Rightarrow (1)$ follow from Theorem \ref{4.5}.

$(1)\Rightarrow (2)$. Let $Y$ be a Banach space and let $T\in \mathcal{L}(X,Y)$. Let $\epsilon>0$ be such that $T^{*}B_{Y^{*}}\subset K+\epsilon B_{X^{*}}, K\subset X^{*}$ is weakly compact. Let $(x_{n})_{n}\in l^{w}_{p}(X)$ and $(x_{n})_{n}\subset B_{X}$. Since $X$ has the $DPP_{p}$, it follows from Theorem \ref{4.5} that
$\lim_{n\rightarrow \infty}\sup_{x^{*}\in K}|<x^{*},x_{n}>|=0$.
Let $c>0$. Then there exists a positive integer $N$ such that $\sup_{x^{*}\in K}|<x^{*},x_{n}>|<c$ for each $n\geq N$.
For each $n\in \mathbb{N}$, pick $y^{*}_{n}\in B_{Y^{*}}$ with $\|Tx_{n}\|=<y^{*}_{n},Tx_{n}>$. Since $T^{*}B_{Y^{*}}\subset K+\epsilon B_{X^{*}}$, then, for each $n\in \mathbb{N}$, there exists $x^{*}_{n}\in K$ such that $\|T^{*}y^{*}_{n}-x^{*}_{n}\|\leq \epsilon$.
Then, for $n\geq N$, we get
\begin{align*}
\|Tx_{n}\|&=<T^{*}y^{*}_{n},x_{n}>\\
&\leq \epsilon+|<x^{*}_{n},x_{n}>|\\
&\leq \epsilon+\sup_{x^{*}\in K}|<x^{*},x_{n}>|\\
&\leq \epsilon+c.
\end{align*}
This yields $\limsup_{n}\|Tx_{n}\|\leq \epsilon+c.$ Since $c>0$ is arbitrary, we obtain $\limsup_{n}\|Tx_{n}\|\leq \epsilon$
and hence $uc_{p}^{1}(T)\leq \epsilon$. This proves $uc_{p}^{1}(T)\leq \omega(T^{*})$.

$(1)\Rightarrow (3)$. Let $(x_{n})_{n}$ be a weakly $p$-summable sequence in $B_{X}$. Let $\epsilon>0$ be such that $A\subset K+\epsilon B_{X^{*}}, K\subset X^{*}$ is weakly compact.
For each $x^{*}\in A$, there exists $z^{*}\in K$ such that $\|x^{*}-z^{*}\|\leq \epsilon$. This yields
$$|<x^{*},x_{n}>|\leq \epsilon+\sup_{x^{*}\in K}|<x^{*},x_{n}>| \quad(n=1,2,...).$$
Since $X$ has the $DPP_{p}$, it follows from Theorem \ref{4.5} that
$\lim_{n\rightarrow \infty}\sup_{x^{*}\in K}|<x^{*},x_{n}>|=0$.
Thus we get $\limsup_{n}\sup_{x^{*}\in A}|<x^{*},x_{n}>|\leq \epsilon$, which completes the proof $(1)\Rightarrow (3)$.

$(1)\Rightarrow (4)$. Let $(x^{*}_{n})_{n}$ be a bounded sequence in $X^{*}$. Let $\epsilon>0$ be such that $(x^{*}_{n})_{n}\subset K+\epsilon B_{X^{*}}, K\subset X^{*}$ is weakly compact. For each $x^{*}_{n}$, there exists $z^{*}_{n}\in K$ such that $\|x^{*}_{n}-z^{*}_{n}\|\leq \epsilon$. Since $K$ is weakly compact, there exists a weakly convergent subsequence $(z^{*}_{k_{n}})_{n}$ of $(z^{*}_{n})_{n}$. By Theorem \ref{5.1}, we see that the sequence
$(z^{*}_{k_{n}})_{n}$ is $\rho^{*}_{p}$-convergent and hence $ca_{\mathcal{F}_{p}}((z^{*}_{k_{n}})_{n})=0$. Note that for any
$F\in \mathcal{F}_{p},F\subset B_{X}$, we have
\begin{align*}
q_{F}(x^{*}_{k_{i}}-x^{*}_{k_{j}})&\leq q_{F}(x^{*}_{k_{i}}-z^{*}_{k_{i}})+q_{F}(z^{*}_{k_{i}}-z^{*}_{k_{j}})+q_{F}(z^{*}_{k_{j}}-x^{*}_{k_{j}})\\
&\leq 2\epsilon+q_{F}(z^{*}_{k_{i}}-z^{*}_{k_{j}}),i,j=1,2,...\\
\end{align*}
This yields
$$ca_{\mathcal{F}_{p}}((x^{*}_{k_{n}})_{n})\leq 2\epsilon+ca_{\mathcal{F}_{p}}((z^{*}_{k_{n}})_{n})=2\epsilon.$$
Hence, we get $\widetilde{ca}_{\mathcal{F}_{p}}((x^{*}_{n})_{n})\leq 2\epsilon$ and then $\widetilde{ca}_{\mathcal{F}_{p}}((x^{*}_{n})_{n})\leq 2\omega((x^{*}_{n})_{n})$.

$(4)\Rightarrow (1)$. Let $(x_{n})_{n}\in l^{w}_{p}(X)$ and let $(x^{*}_{n})_{n}$ be weakly null in $X^{*}$.
By (4), we get $\widetilde{ca}_{\mathcal{F}_{p}}((x^{*}_{n})_{n})=0.$ A classical diagonal argument yields a subsequence
$(x^{*}_{k_{n}})_{n}$ of $(x^{*}_{n})_{n}$ which is $\rho^{*}_{p}$-Cauchy. By the completeness of the topology $\rho^{*}_{p}$, we see that the subsequence $(x^{*}_{k_{n}})_{n}$ is $\rho^{*}_{p}$-convergent. Since $(x^{*}_{n})_{n}$ is weakly null,
$(x^{*}_{k_{n}})_{n}$ is $\rho^{*}_{p}$-null. Since $(x_{n})_{n}$ is weakly $p$-summable, one has
$$|<x^{*}_{k_{n}},x_{k_{n}}>|\leq \sup_{i}|<x^{*}_{k_{n}},x_{i}>|\rightarrow 0 \quad(n\rightarrow \infty).$$
Then Theorem \ref{4.2} gives (1).

$(1)\Rightarrow (5)$. Let $c>\omega(A)$. Then there exists a weakly compact subset $K$ of $X^{*}$ such that $\widehat{d}(A,K)<c$. Since $X$ has the $DPP_{p}$, it follows from Theorem \ref{4.5} that $\chi_{m}^{p}(K)=0$. Let $\epsilon>0$ and $L\in\mathcal{F}_{p},L\subset B_{X}$. Then there exists a finite subset $F\subset K$ such that $\widehat{d}(K|_{L},F|_{L})<\epsilon$, so $\chi(A|_{L})\leq c+\epsilon$. Since $\epsilon>0$ is arbitrary, we get $\chi(A|_{L})\leq c$. By (\ref{3}), we get $\chi_{0}(A|_{L})\leq 2c$. This implies that $\chi_{m}^{p}(A)\leq 2c,$ which completes the proof.

\end{proof}

The following quantitative version obviously strengthens the Dunford-Pettis property of order $p$.
\begin{thm}\label{5.3}
Let $X$ be a Banach space and $1<p<\infty$. The following are equivalent:
\item[(1)]There is $C>0$ such that $uc_{p}^{6}(T)\leq C\cdot wk_{X^{*}}(T^{*})$ for every operator $T$ from $X$ into any Banach space $Y$;
\item[(2)]There is $C>0$ such that $uc_{p}^{6}(T)\leq C\cdot wk_{X^{*}}(T^{*})$ for every operator $T$ from $X$ into $l_{\infty}$;
\item[(3)]There is $C>0$ such that $\eta_{p}(A)\leq C\cdot wk_{X^{*}}(A)$ for each bounded subset $A$ of $X^{*}$;
\item[(4)]There is $C>0$ such that $uc_{p}^{6}(T)\leq C\cdot wk_{Y}(T)$ for every operator $T$ from $X$ into any Banach space $Y$;
\item[(5)]There is $C>0$ such that $uc_{p}^{6}(T)\leq C\cdot wk_{l_{\infty}}(T)$ for every operator $T$ from $X$ into $l_{\infty}$.
\end{thm}
\begin{proof}
The implication $(1)\Rightarrow (2)$ is trivial with the same constant.

$(2)\Rightarrow (3)$. Assume that there is $C>0$ such that $uc_{p}^{6}(T)\leq C\cdot wk_{X^{*}}(T^{*})$ for every operator $T$ from $X$ into $l_{\infty}$. We'll show that (3) holds with the constant $32C$. Let $A$ be a bounded subset of $X^{*}$. We may assume that $\eta_{p}(A)>0$. Let us fix any $0<\epsilon<\eta_{p}(A)$. By the definition of $\eta_{p}(A)$, there exist a sequence $(x^{*}_{n})_{n}$ in $A$ and a sequence $(x_{n})_{n}$ in $l^{w}_{p}(X)$ with $\|(x_{n})_{n}\|^{w}_{p}\leq 1$ such that $|<x^{*}_{n},x_{n}>|>\epsilon$ for each $n\in \mathbb{N}$. Let us define an operator $S:l_{1}\rightarrow X^{*}$ by $$S((\alpha_{n})_{n})=\sum_{n}\alpha_{n}x^{*}_{n}, \quad (\alpha_{n})_{n}\in l_{1}.$$ As in the proof of Theorem 5.4 in \cite{KKS}, the set $S(B_{l_{1}})$ is contained in the closed absolutely convex hull of $(x^{*}_{n})_{n}$ and so $wk_{X^{*}}(S)\leq 2wk_{X^{*}}((x^{*}_{n})_{n}).$
Let $T=S^{*}J_{X}:X\rightarrow l_{\infty}$. By (2) and (\ref{1}), we get $uc_{p}^{6}(T)\leq C\cdot wk_{X^{*}}(T^{*})$. Thus
\begin{align*}
\epsilon &\leq \limsup_{n}|<x^{*}_{n},x_{n}>| \leq \limsup_{n}\|Tx_{n}\|\\
&\leq uc_{p}^{6}(T)\leq C\cdot wk_{X^{*}}(T^{*})\\
&\leq 4C\cdot wk_{l_{\infty}}(T)\leq 4C\cdot wk_{l_{\infty}}(S^{*})\\
&\leq 16C\cdot wk_{X^{*}}(S)\leq 32C\cdot wk_{X^{*}}((x^{*}_{n})_{n})\\
&\leq 32C\cdot wk_{X^{*}}(A)\\
\end{align*}
Since $\epsilon<\eta_{p}(A)$ is arbitrary, we get the assertion (3).

$(3)\Rightarrow (1)$. Let us suppose that (3) holds with a constant $C>0$.
Let $T\in \mathcal{L}(X,Y)$.
Let $(x_{n})_{n}\in l^{w}_{p}(X)$ with $\|(x_{n})_{n}\|^{w}_{p}\leq 1$. For each $n\in \mathbb{N}$, pick $y^{*}_{n}\in B_{Y^{*}}$ so that $\|Tx_{n}\|=<y^{*}_{n},Tx_{n}>$.
Applying (3) to $A=(T^{*}y^{*}_{n})_{n}$, we get
\begin{align*}
\limsup_{n}\|Tx_{n}\|&=\limsup_{n}|<T^{*}y^{*}_{n},x_{n}>|\\
&\leq \limsup_{n}\sup_{x^{*}\in A}|<x^{*},x_{n}>|\leq \eta_{p}(A)\\
&\leq C\cdot wk_{X^{*}}(A)\leq C\cdot wk_{X^{*}}(T^{*}),\\
\end{align*}
which yields $uc_{p}^{6}(T)\leq C\cdot wk_{X^{*}}(T^{*})$.

Finally, the equivalences of $(1)\Leftrightarrow (4)$ and $(2)\Leftrightarrow (5)$ follow from estimate (\ref{1}).

\end{proof}
It should be mentioned that the assertion (3) of Theorem \ref{5.3} is a quantitative version of Theorem \ref{4.5}.
\begin{defn}
We say that a Banach space $X$ has the \textit{quantitative Dunford-Pettis property of order $p$} if $X$ satisfies the equivalent conditions of Theorem \ref{5.3}.
\end{defn}
The following Theorem \ref{5.5} is a quantitative version of Corollary \ref{4.3}. To prove it, we need a simple lemma.
\begin{lem}\label{5.4}
Let $X$ be a closed subspace of a Banach space $Y$ and let $A$ be a bounded subset of $X$. Then
\begin{equation}\label{2}
wk_{Y}(A)\leq wk_{X}(A)\leq 2wk_{Y}(A).
\end{equation}
\end{lem}
\begin{proof}
We can identify $X^{**}$ with $X^{\bot\bot}\subset Y^{**}$. Under this identification, the $weak^{*}$ closure of $A$ in $X^{**}$ is equal to the $weak^{*}$ closure of $A$ in $Y^{**}$. This yields the left inequality immediately. To prove the right inequality of (\ref{2}), let us fix any $c>wk_{Y}(A)$. Take any $y^{**}\in \overline{A}^{w^{*}}$. Then there exists $y\in Y$ such that $\|y^{**}-y\|\leq c$. Choose $y^{*}\in X^{\bot}$ with $\|y^{*}\|=1$ so that $d(y,X)=|<y^{*},y>|$. Then we get
$$d(y^{**},X)\leq \|y^{**}-y\|+d(y,X)\leq c+|<y^{*},y>|=c+|<y^{*},y^{**}-y>|\leq 2c.$$
Thus $wk_{X}(A)\leq 2c$. By the arbitrariness of $c>wk_{Y}(A)$, we obtain $wk_{X}(A)\leq 2wk_{Y}(A)$.

\end{proof}

\begin{thm}\label{5.5}
If $X^{**}$ has the quantitative Dunford-Pettis property of order $p$, then so is $X$. More precisely,
\item[(a)]If $X^{**}$ satisfies one of the conditions (1),(2),(4) and (5) of Theorem \ref{5.3} with a given constant $C$, then $X$ satisfies the respective condition of Theorem \ref{5.3} with $16C$;
\item[(b)]If $X^{**}$ satisfies the condition (3) of Theorem \ref{5.3} with a given constant $C$, then $X$ satisfies the respective condition (3) of Theorem \ref{5.3} with $C$.
\end{thm}
\begin{proof}
The assertion (a) follows immediately from the inequality (\ref{1}) and the easy fact that $uc_{p}^{6}(T)\leq uc_{p}^{6}(T^{**})$ for each operator $T$. The assertion (b) is a direct consequence of (\ref{2}).
\end{proof}

{\bf Acknowledgements.} This work is done during the first author's visit to Department of Mathematics, Texas A\&M University. We would like to thank Professor W. B. Johnson for helpful discussions and comments.

\end{document}